\documentclass[11pt, reqno]{amsart}
\usepackage{amssymb}
\usepackage{verbatim}
\usepackage[vcentermath]{youngtab}
\usepackage{float}
\usepackage{tikz-cd}
\usepackage{mathtools}
\usepackage[margin =1.29in]{geometry}

\usepackage{times}
\usepackage[T1]{fontenc}
\usepackage{mathrsfs}
\usepackage{epsfig}
\usepackage{color}

\newtheorem{thm}{Theorem}
\newtheorem{lemma}[thm]{Lemma}

\theoremstyle{remark}

\theoremstyle{definition}

\makeatletter
\newcommand{\subalign}[1]{%
  \vcenter{%
    \Let@ \restore@math@cr \default@tag
    \baselineskip\fontdimen10 \scriptfont\tw@
    \advance\baselineskip\fontdimen12 \scriptfont\tw@
    \lineskip\thr@@\fontdimen8 \scriptfont\thr@@
    \lineskiplimit\lineskip
    \ialign{\hfil$\m@th\scriptstyle##$&$\m@th\scriptstyle{}##$\crcr
      #1\crcr
    }%
  }
}
\makeatother

\def\FF{\mathbb{F}}

\def\gal{\mathrm{Gal}}

\begin{document}

\title{Normal Elements in Finite Fields}

\author{Trevor Hyde}
\address{Dept. of Mathematics\\
University of Michigan \\
Ann Arbor, MI 48109-1043\\
}
\email{tghyde@umich.edu}

\date{September 6th, 2018}

\maketitle

If $L/K$ is a finite Galois field extension with Galois group $G$, then $\alpha \in L$ is called a \textbf{normal element} if the $G$-orbit of $\alpha$ forms a basis of $L$ as a vector space over $K$. The normal basis theorem \cite[Thm. 13.1]{Lang} asserts that every finite Galois extension has a normal element.

If $\FF_{q^n}/\FF_q$ is an extension of finite fields, then there are finitely many normal elements in $\FF_{q^n}$. We give a simple proof of a formula for $N_n(q)$ the number of normal elements in $\FF_{q^n}/\FF_q$.

\begin{thm}
\label{thm main}
Let $p = \mathrm{char}(\FF_q)$ and write $n = dp^m$ where $p$ does not divide $d$. For $e$ coprime to $q$ let $o_e(q)$ denote the multiplicative order of $q$ modulo $e$, and let $\varphi$ be the Euler totient function. Then
\[
    N_n(q) = q^n\prod_{e\mid d}\left(1 - \frac{1}{q^{o_e(q)}}\right)^{\varphi(e)/o_e(q)}.
\]
\end{thm}

Proofs of Theorem \ref{thm main} appear in \cite{Akbik}, \cite[Cor. 2.4.7]{Gao}, \cite[Sec. 1]{Lenstra}, and \cite[Chp. 1, Thm. 12]{Ore}. Ore \cite[Pg. 251]{Ore} attributes the problem of determining $N_n(q)$ to Eisenstein and the first complete solution to Hensel \cite{Hensel}. At the end of this note we discuss the relation between our proof and those mentioned above.

Our proof is based on Theorem \ref{lem torsor} which describes a general relation between normal elements and units in the Galois group ring. This result was communicated to us by O'Desky and Rosen \cite{Odesky}. While Theorem \ref{lem torsor} and its application to Theorem \ref{thm main} may be known to experts, we did not find this in the literature. The aim of this note is to derive Theorem \ref{thm main} with a simple, direct argument and to bring Theorem \ref{lem torsor} to a wider audience.

\begin{thm}[{\cite{Odesky}}]
\label{lem torsor}
If $L/K$ is a finite Galois extension with Galois group $G$, then the units $K[G]^\times$ in the group algebra for $G$ over $K$ act freely and transitively on the set of normal elements in $L/K$. In other words, the normal elements in $L/K$ form a torsor for $K[G]^\times$.
\end{thm}

\begin{proof}
First suppose that $u \in K[G]^\times$ and $\alpha$ is a normal element. We claim that $\beta := u\alpha$ is a normal element. If for some $b_g \in K$
\[
    0 = \sum_{g\in G}b_g g\beta = \sum_{g\in G} b_g gu\alpha,
\]
then $\alpha$ normal implies that $\sum_{g\in G}b_g gu = 0$ in $K[G]$. Dividing by $u$ on the right gives $\sum_{g\in G}b_g g = 0$, hence $b_g = 0$ for all $g$. Thus the $G$-orbit of $\beta$ is linearly independent, hence $\beta$ is normal. Furthermore, by the normality of $\alpha$, we see that $u\alpha = \alpha$ implies $u = 1$. This shows that $K[G]^\times$ acts freely on the normal elements in $L/K$.

Next we show $K[G]^\times$ acts transitively on normal elements. Suppose that $\alpha$ and $\beta$ are both normal elements in $L/K$. Then for some $a_g, b_g \in K$ we have
\[
    \alpha = \sum_{g\in G} a_g g\beta \hspace{.75in} \beta = \sum_{g\in G}b_g g\alpha.
\]
If $u := \sum_{g\in G}a_g g$ and $v := \sum_{g\in G}b_g g$, then $\alpha = u\beta$ and $\beta = v\alpha$. Thus $uv\alpha = \alpha$ and $vu\beta = \beta$. Since $\alpha$ and $\beta$ are normal elements this implies that $uv = 1 = vu$, hence $u$ is a unit in $K[G]^\times$. Therefore every normal element $\beta$ is a $K[G]^\times$ multiple of $\alpha$.
\end{proof}

A weaker version of Theorem \ref{lem torsor} appears implicitly in Suwa \cite[Cor. 1.7]{Suwa} where it is traced back to an argument of Serre \cite[Chp. IV, Prop. 7]{Serre}. Their statement is an equivalence between the existence of a normal basis of a Galois algebra and of a certain pull-back diagram involving units in a group scheme associated to the Galois group ring.

Lemma \ref{lem phi} below is the function field analog of the classic formula for Euler's totient function
\[
    \varphi(n) = n\prod_{p\mid n}\left(1 - \frac{1}{p}\right),
\]
where the product on the right is taken over all prime divisors of $n$ without multiplicity. Note that $\varphi(n)$ may be defined as the number of multiplicative units modulo $n$. Lemma \ref{lem phi} is well-known, see \cite[Prop. 1.7]{Rosen} for example. We give a proof for completeness.

\begin{lemma}
\label{lem phi}
If $f(x) \in \FF_q[x]$ is non-constant, then
\[
    |(\FF_q[x]/(f))^\times| = q^{\deg(f)}\prod_{p(x) \mid f(x)}\left(1 - \frac{1}{q^{\deg(p)}}\right),
\]
where the product is taken over all monic irreducible factors of $f(x)$ without multiplicity.
\end{lemma}

\begin{proof}
Let $\varphi(f) := |(\FF_q[x]/(f))^\times|$. Then the probability of a uniformly random element of $\FF_q[x]/(f)$ being a unit is, on one hand, simply $\varphi(f)/q^{\deg(f)}$. On the other hand, $u(x) \in (\FF_q[x]/(f))^\times$ is equivalent to $u(x)$ not being divisible by any irreducible $p(x)$ dividing $f(x)$, and these events are independent by the Chinese Remainder Theorem. Hence
\[
    \frac{\varphi(f)}{q^{\deg(f)}} = \prod_{p(x) \mid f(x)}\left(1 - \frac{1}{q^{\deg(p)}}\right).\qedhere
\]
\end{proof}

\begin{proof}[Proof of Theorem \ref{thm main}]
The Galois group of $\FF_{q^n}/\FF_q$ is cyclic of order $n$ generated by the Frobenius auotomorphism $\sigma: a \mapsto a^q$. Therefore the group algebra $\FF_q[\langle\sigma\rangle]$ is naturally isomorphic to $\FF_q[x]/(x^n - 1)$ by $\sigma \mapsto x$. Theorem \ref{lem torsor} implies that the number of normal elements $N_n(q)$ is equal to the number of units in $\FF_q[\langle\sigma\rangle]$, hence by Lemma \ref{lem phi}
\begin{equation}
\label{eqn prod}
    N_n(q) = |(\FF_q[x]/(x^n-1))^\times| = q^n \prod_{p(x) \mid x^n - 1}\left(1 - \frac{1}{q^{\deg(p)}}\right).
\end{equation}
If $n = dp^m$ where $p$ does not divide $d$, then $x^n - 1 = (x^d - 1)^{p^m}$ in $\FF_q[x]$, hence the product \eqref{eqn prod} may be taken over irreducibles $p(x) \mid x^d - 1$. By Galois theory these irreducible factors correspond to orbits of Frobenius on the $d$th roots of unity. The orbit of a primitive $e$th root of unity has length $o_e(q)$, the multiplicative order of $q$ modulo $e$, and there are $\varphi(e)/o_e(q)$ such orbits. Hence
\[
    N_n(q) = q^n \prod_{e\mid d}\left(1 - \frac{1}{q^{o_e(q)}}\right)^{\varphi(e)/o_e(q)}.\qedhere
\]
\end{proof}

The proofs of Theorem \ref{thm main} in \cite{Akbik, Hensel, Lenstra, Ore} use additive polynomials to count normal elements in $\FF_{q^n}/\FF_q$. Recall that a polynomial $f(x) \in \FF_q[x]$ is \emph{additive} if $f(x + y) = f(x) + f(y)$ or equivalently if $f(x) = \sum_{i=0}^d a_i x^{q^i}$. Additive polynomials are an essentially positive characteristic phenomenon and thus this approach gives the impression that the enumeration in Theorem \ref{thm main} hinges on some special feature of positive characteristic fields. However, the ring of additive polynomials is isomorphic to the Galois group ring $\FF_q[\langle \sigma\rangle]$, and the latter generalizes to Galois extensions in any characteristic. The important underlying structure is the free transitive action of the units in the Galois group ring on normal elements (Theorem \ref{lem torsor}) and the structure of finite fields only comes in to count $(\FF_q[x]/(x^n-1))^\times$.

For example, if $L/K$ is any degree $n$ cyclic Galois extension, then each choice of normal element in $L/K$ and generator of the Galois group provides an explicit bijection between all normal elements and units in the algebra $K[x]/(x^n - 1)$, a product of cyclotomic extensions of $K$.

Ore \cite[Chp. 1]{Ore} proves Theorem \ref{thm main} by studying the minimal additive polynomial associated to each element of $\FF_{q^n}$ and observing that normal elements are precisely those whose minimal additive polynomial is $x^{q^n} - x$. He then uses an inclusion-exclusion argument to arrive at the product formula \eqref{eqn prod}. Akbik \cite{Akbik} independently came to essentially the same proof nearly 60 years later. Lenstra and Schoof \cite[Sec. 1]{Lenstra} give an exposition of Ore's proof in their work on primitive normal bases; the ideas behind our proof appear there between the lines. They describe Ore's results as pertaining to the Galois module structure of $\FF_{q^n}/\FF_q$, but neither explicitly state that the ring of additive polynomials is the group ring of $\gal(\FF_{q^n}/\FF_q)$ nor refer to Theorem \ref{lem torsor} directly. They do allude to the general connection between normal elements and units in the Galois group ring in their assertion \cite[(1.15) Pg. 221]{Lenstra}.

Perlis \cite[Lem. 1]{Perlis} gives a criterion for an element in $\FF_{q^n}$ to generate an normal basis which is equivalent to Ore's characterization in terms of minimal additive polynomials. Gao \cite[Pg. 18]{Gao} notes that Perlis's result can be used to count the number normal elements. Perlis \cite[Thm. 1]{Perlis} also shows that normal elements may be detected by the non-vanishing of their trace, and Gao \cite[Cor. 2.4.7]{Gao} uses this to give another enumeration of normal elements.

\subsection*{Acknowledgments}
We thank Andrew O'Desky for sharing Theorem \ref{lem torsor} and clarifying its relation to the work of Suwa and Serre. We thank Julian Rosen and Mike Zieve for bringing references to our attention, in particular \cite{Gao, Lenstra, Ore, Perlis}. We thank Bob Lutz and Andrew O'Desky for comments on an earlier draft. Finally we are grateful to Jeff Lagarias for helpful feedback and encouragement.


\begin{thebibliography}{9}
\bibitem{Akbik}
S. Akbik, Normal generators of finite fields, \emph{J. Number Theory}, \textbf{41}, (1992), 146-149.

\bibitem{Gao}
S. Gao, Normal bases over finite field, Dissertation, \emph{University of Waterloo}, (1993).

\bibitem{Hensel}
K. Hensel, \"Uber die Darstellun der Zahlen eines Gattungsbereiches f\"ur einen beliebigen Primdivisor, \emph{Journal f\"ur Mathematik}, \textbf{103}, (1888), 230-237.

\bibitem{Lang}
S. Lang, \emph{Algebra}, \textbf{211}, Springer Science \& Business Media, (2002).

\bibitem{Lenstra}
H. W. Lenstra, R. Schoof, Primitive normal bases for finite fields, \emph{Math. Comp.}, \textbf{48}, No. 177, (1987), 217-231.

\bibitem{Odesky}
A. O'Desky, J. Rosen, Group rings, normal bases, and \'{e}tale algebras, in preparation.

\bibitem{Ore}
O. Ore, Contributions to the theory of finite fields, \emph{Trans. Amer. Math. Soc.}, \textbf{36}, No. 2, (1934), 243-274.

\bibitem{Perlis}
S. Perlis, Normal bases of cyclic fields of prime-power degree, \emph{Duke Math. J.} \textbf{9}, No. 3, (1942), 507-517.

\bibitem{Rosen}
M. Rosen, \emph{Number Theory in Function Fields}, \textbf{210}, Springer Science \& Business Media, (2002).

\bibitem{Serre}
J.-P. Serre, \emph{Algebraic groups and class fields}, \textbf{117}, Springer Science \& Business Media, (2012).

\bibitem{Suwa}
N. Suwa, Around Kummer theories, RIMS K\^oky\^uroku Bessatu B12, (2009), 115-148.
\end{thebibliography}
\end{document}